\documentclass[a4paper,11pt]{amsproc}
\usepackage{amssymb}
\usepackage{amsmath}
\usepackage{amsthm,lmodern} 
\usepackage{tikzsymbols}
\usepackage[foot]{amsaddr}

\usepackage{framed, booktabs, float}
\usepackage{xcolor}
\usepackage[backref]{hyperref}
\hypersetup{
    colorlinks,
    linkcolor={red!60!black},
    citecolor={green!60!black},
    urlcolor={blue!60!black}
}
\usepackage{enumerate}
\usepackage{mathtools}
\usepackage{url}
\usepackage[T1]{fontenc}
\usepackage{geometry}
\usepackage{tikz}
\usetikzlibrary{3d}
\usetikzlibrary{calc}

\DeclareMathOperator{\Sym}{Sym}

\theoremstyle{plain}
\newtheorem{Theorem} {Theorem} [section]
\newtheorem{Proposition} [Theorem] {Proposition}
\newtheorem{Lemma} [Theorem] {Lemma}
\newtheorem{Corollary} [Theorem] {Corollary}

\newtheorem{Problem} [Theorem] {Problem}
\newtheorem{Remark} [Theorem] {Remark}

\theoremstyle{definition}

\newcommand{\HG}{{\mathrm{HG}}}
\newcommand{\HGP}{\mathrm{HG_{P}}}
\newcommand{\Prob}{\mathbb{P}}

\newcommand{\Ex}{\textnormal{Ex}}

\renewcommand{\th}{\textsuperscript{th} }

\title{Hat guessing with proper colorings}

\author
[Adriaensen et al.]
{Sam Adriaensen$^1$}
\address{$^1$Department of Mathematics and Data Science, Vrije Universiteit Brussel, Belgium.}

\author
{Peter Bentley$^2$}
\address{$^2$Department of Mathematics, 
William \& Mary, 
Virginia,
United States.}

\author
{Anurag Bishnoi$^3$}
\address{$^3$Delft Institute of Applied Mathematics, Delft University of Technology, Netherlands.}

\author
{Wouter Cames van Batenburg$^4$}
\address{$^4$D\'epartement d'Informatique, Universit\'e libre de Bruxelles, Belgium.}

\author
{Michael Kreiger$^5$}
\address{$^5$Department of Mathematics, 
University of Illinois Urbana-Champaign, 
Illinois, United States.}

\author
{Lars van der Kuil$^6$}
\address{$^6$Leiden Institute of Advanced Computer Science, Leiden University, Netherlands.}

\author
{Saptarshi Mandal$^7$}
\address{$^7$Theoretical Statistics and Department of Mathematics and Statistics, IIT Kanpur, India.}

\author
{Anurag Ramachandran$^8$}
\address{$^8$Theoretical Statistics and Mathematics Unit, Indian Statistical Institute, Bangalore, India.}

\author
{James Tuite$^{9,10}$}
\address{$^9$School of Mathematics and Statistics, Open University, Milton Keynes, UK. }
\address{$^{10}$Department of Informatics and Statistics, Klaip\.{e}da University, Klaip\.{e}da, Lithuania.}

\email{Sam.Adriaensen@vub.be}
\email{pjbentley@wm.edu}
\email{A.Bishnoi@tudelft.nl}
\email{w.p.s.camesvanbatenburg@gmail.com}
\email{mkreiger.math@gmail.com}
\email{larsvanderkuil@gmail.com}
\email{saptarshim23@iitk.ac.in}
\email{nt5.anuragramachandran@gmail.com}
\email{james.t.tuite@open.ac.uk}

\begin{document}

\begin{abstract}
We initiate the study of the hat guessing number of a graph where the adversary is only allowed to provide a proper coloring of the graph. 
This is the largest number $q$ for which there is a guessing strategy on each vertex that only depends on its neighborhood, such that for every proper coloring of the graph with $q$ colors at least one vertex guesses its color correctly.
In this variation, we prove that the hat guessing number of the complete graph on $n$ vertices is $2n - 1$, which is roughly twice the ordinary hat guessing number of the complete graph. 
Our winning strategy is related to finding perfect matchings between the middle layers of the boolean poset of dimension $2n - 1$.
We prove that the hat guessing number of all trees on $n \geq 3$ vertices is equal to $4$. 
We derive general upper bounds in terms of the number of vertices, chromatic number, and maximum degree, and obtain improved bounds for book graphs.
Using our results and an ILP formulation of the problem, we determine the exact hat guessing number for all graphs on at most $4$ vertices, give bounds on graphs on $5$ vertices, and suggest some open problems.
 \end{abstract}

\maketitle

\section{Introduction}

In the hat guessing game \cite{Winkler2001} on a graph $G=(V,E)$, each vertex $v\in V$ hosts a player. 
An adversary places on every player a hat whose color is drawn from a fixed set $[q]=\{1,2,\dots,q\}$. 
Each player sees the hat colors of their neighbors, that is, the colors on $N(v)$, but not their own, and all players must simultaneously output a guess for their own hat color. Before placing the hats, the players should agree on a strategy: for each $v\in V$, a function 
\[
f_v : [q]^{N(v)} \to [q].
\]
A strategy is said to be winning for the players if at least one player guesses correctly, no matter how the adversary assigns the $q$ colors to the vertices. The \emph{hat guessing number} $\HG(G)$ is the largest $q$ for which a winning strategy exists on $G$. 
This parameter has been studied for various families of graphs \cite{Alon2020}, such as planar graphs \cite{Bradshaw2022, latyshev2023}, cycles \cite{Szczechla2017} and book graphs \cite{He2022}. 
While we know some bounds on this parameter, for most graphs finding the exact, or even the asymptotic value, remains open. 
\newline

In this paper, we study a new variant of this game in which the adversary is only allowed to place hats whose colors form a \textit{proper coloring} of the graph $G$, that is, no two neighboring players can get the same colored hat. 
We define $\HGP(G)$ as the largest number $q$ of colors such that for graph $G$ there exists a strategy where for every proper coloring of $G$ using $q$ colors, at least one vertex guesses correctly. 
Clearly, $\HG(G) \leq \HGP(G)$.
We will show that this bound is far from tight. 
For cliques $K_n$, it is known that the ordinary hat guessing number is equal to $n$ \cite{Feige2004}.
We prove the following for the proper hat guessing number of cliques. 

\begin{Theorem}\label{Thm:CompleteGraphs}
 For every integer $n \geq 2$,
 \[\HGP(K_n) = 2n - 1.\]
\end{Theorem}

We prove the lower bound using a strategy based on perfect matchings between the set of $(n - 1)$ element subsets and $n$-element subsets of $[2n - 1]$, and the upper bound from the following more general bound on all graphs.

\begin{Lemma}\label{Lem:ChromaticNumber}
    For every graph $G$ on $n$ vertices 
    \[\HGP(G) \leq n + \chi(G) - 1.\]
\end{Lemma}

Since $\chi(G) \leq n$ for any graph on $n$ vertices, with equality if and only if $G = K_n$, we can conclude that $\HGP(G) = 2n - 1$ if and only if $G = K_n$.

In addition, we upper bound $\HGP(G)$ using the maximum degree $\Delta(G)$.
For the classical hat guessing number, a straightforward application of the Lov\'asz Local Lemma proves that $\HG(G) < e \Delta(G)$, see e.g.\ \cite[Theorem 1.5]{Alon2020}.
For the proper hat guessing number, using only proper colorings can create extra dependencies between the hat colors of non-neighboring vertices, and we have to be more careful in our analysis.
Using a Lopsided Lov\'asz Local Lemma, we prove the following bound.

\begin{Theorem}
\label{thm:HGPlinear}
Let $G$ be a graph of maximum degree $\Delta\ge 2$.  Then
$$ \HGP(G) < \Delta+ \left\lceil \frac{\Delta^\Delta}{(\Delta-1)^{\Delta-1}}\right\rceil  <(e+1)\Delta.$$
\end{Theorem}

Another family of graphs for which we can determine the number exactly are trees. 
Note that a tree on $2$ vertices is isomorphic to $K_2$ and hence has the proper hat guessing number equal to $3$. 
We show that the unique tree on $3$ vertices has proper hat guessing number equal to $4$, and then inductively deduce the following.

\begin{Theorem}
 \label{Thm:Trees}
For every tree $T$ on at least $3$ vertices, 
\[\HGP(T) = 4.\]
\end{Theorem}

For the Book graphs $B_{k, n}$, obtained by connecting all vertices of a clique of size $k$ to an independent set of size $n$ we prove the following.

\begin{Theorem}\label{Thm:BookGraphUpperBound}
 If $n > e^2 k$, then
 \[
  \HGP(B_{k,n}) < \left( \frac{e^2 k}{n} \right)^{1/n}n + k + 1.
 \]
\end{Theorem}

In Section~\ref{sec:general_bounds}, we prove some bounds on our variation of the hat guessing number in terms of the chromatic number of the graph, and the maximum degree, namely Lemma~\ref{Lem:ChromaticNumber} and Theorem~\ref{thm:HGPlinear}. 
In Section~\ref{sec:complete_graphs}, we prove Theorem~\ref{Thm:CompleteGraphs}.
Theorems~\ref{Thm:Trees} and \ref{Thm:BookGraphUpperBound} are then proved in Section~\ref{sec:trees} and \ref{sec:book_graphs}, respectively. 
In Section~\ref{sec:small_graphs} we determine the exact hat guessing number of all graphs on at most $4$ vertices, and give bounds for graphs on $5$ vertices.

\section{General bounds}
\label{sec:general_bounds}

In this section, we give some general upper bounds on $\HGP(G)$ in terms of other parameters of $G$.
We start with a simple but useful observation.

\begin{Lemma}\label{Lem:Subgraphs}
 If $H$ is a subgraph of $G$, then $\HGP(H) \leq \HGP(G)$.
\end{Lemma}

Let $\chi(G)$ denote the chromatic number of the graph $G$, i.e.\ the minimum number of colors required for a proper coloring of $G$.

\bigskip

{\scshape Lemma \ref{Lem:ChromaticNumber}.}
{\it 
 For every graph $G$ on $n$ vertices 
 \[
  \HGP(G) \leq n + \chi(G) - 1.
 \]}

\begin{proof}
 Suppose that there is a winning strategy for the proper coloring hat guessing game on $G$ with $q$ colors.
 We restrict our focus to the proper colorings of $G$ that use only $\chi(G)$ colors in $[q]$.
 Given any such coloring $c$ of $G$ and a vertex $v$ of $G$, the neighbors of $v$ are colored by at most $\chi(G)-1$ different colors.
 Hence, there are at least $q-\chi(G)+1$ possible colors for $v$, and for only one of these colors does $v$ guess correctly.
 Therefore, if we uniformly randomly choose such a coloring $c$, then $v$ has a probability of at most $\frac{1}{q-\chi(G)+1}$ to guess correctly.
 This means that the expected number of correct guesses is at most $\frac{n}{q-\chi(G)+1}$.
 On the other hand, since the strategy is winning, for every coloring $c$, at least one vertex guesses correctly.
 Thus, $\frac{n}{q-\chi(G)+1} \geq 1$, or equivalently $q \leq n + \chi(G) - 1$.
\end{proof}

\begin{Remark}
 \label{Rmk:ChromaticNumber}
 Following the proof of Lemma \ref{Lem:ChromaticNumber}, one sees that the bound $\HGP(G) \leq n + \chi(G) - 1$ can only be tight if for every proper coloring of $G$ using $\chi(G)$ colors, for every vertex of $G$, its neighbors are colored using $\chi(G)-1$ colors.
 In other words, if there exists at least one coloring of $G$ with $\chi(G)$ colors for which at least one vertex sees at most $\chi(G)-2$ colors in its neighborhood, then
 \[
  \HGP(G) \leq n + \chi(G) - 2.
 \]
\end{Remark}

We can also give a bound that relates the proper hat guessing number of $G$ to the ordinary hat guessing number, that outperforms the above bound in some cases, and might potentially be used to lower bound $\HG(G)$ using $\HGP(G)$.

\begin{Proposition}
 \label{Lem:BoundFromOrdinaryHG}
 For any graph $G$, $\HGP(G) < \chi(G)(\HG(G) + 1)$.
\end{Proposition}

\begin{proof}
 Let $q = \HG(G)$, and suppose for the sake of contradiction that $\HGP(G) \geq \chi(G)(q+1)$.
 We use the set $[\chi(G)] \times [q+1]$ as the set of $\chi(G)(q+1)$ colors.
 Let $\{ f_w : w \in V(G)\}$ be a winning strategy for the proper coloring hat guessing game on $G$ using this set of colors.
 Fix a proper coloring $c_1: V(G) \to [\chi(G)]$ of $G$.
 We now construct a winning strategy for the ordinary hat guessing game on $G$ using $q+1$ colors.
 For any (not necessarily proper) coloring $c_2$ of $G$ with colors in $[q+1]$, define the proper coloring
 \[
  c_1 \times c_2: V(G) \to [\chi(G)] \times [q+1]:
  v \mapsto (c_1(v), c_2(v))
 \]
 of $G$.
 Let $g_w$ map the coloring $c_2$ to the second entry of $f_w(c_1 \times c_2)$, i.e.\ the entry in $\{1,\dots,q+1\}$.
 Since the $f_w$ strategies are winning, for every coloring $c_2$, there exists a vertex $v \in V(G)$ such that $f_v(c_1 \times c_2) = (c_1(v),c_2(v))$, which implies that $g_v(c_2) = c_2(v)$.
 But then the $g_w$'s are a winning strategy for the ordinary hat guessing game on $G$ using more than $q = \HG(G)$ colors, which is a contradiction.
\end{proof}

The upper bound in Proposition~\ref{Lem:BoundFromOrdinaryHG} can give good bounds on $\HGP(G)$ for various graphs with bounded chromatic numbers. 
In \cite{Knierim2023}, several classes of graphs with bounded chromatic number have been shown to have bounded ordinary hat guessing number.
Our result implies that these graphs should also have a bounded proper hat guessing number. 
In particular, if $G$ is an outerplanar graph, then we can conclude from \cite[Theorem 1.4]{Knierim2023} and Proposition~\ref{Lem:BoundFromOrdinaryHG} that $\HGP(G) \leq 3\cdot41 - 1= 122.$

\bigskip 

We conclude this section by proving Theorem~\ref{thm:HGPlinear}.
We need a Lopsided Lov\'asz Local
Lemma, which is a variant of the usual Lov\'asz Local Lemma that relaxes the conditions of independence. Erd\H{o}s and Spencer~\cite{ErdosSpencer} proved one of the first versions, in terms of positive correlations.
Here, we use the following symmetric form given in
Knuth~\cite[p.~82, Section 7.2.2.2, Lemma~L and Theorem~J]{KnuthBook}, which uses a direct estimate on conditional probabilities as input. This variant can also be derived from more general work of Scott and Sokal~\cite[Theorem~4.1, Definition~2.14, and Corollary~5.7]{ScottSokal}.

\begin{Lemma}[Symmetric Lopsided Lov\'{a}sz Local Lemma]
\label{lem:LLLL}
Let $B_1,\dots,B_m$ be bad events in a probability space, and let $D$ be a
graph on $[m]$ of maximum degree at most $\Delta\ge 2$. Let $p\le \frac{(\Delta-1)^{\Delta-1}}{\Delta^\Delta}$. Suppose that, for every
$i\in[m]$ and every set
$S\subseteq [m]\setminus N_D[i]$
for which $\Prob(\bigcap_{j\in S}\overline{B_j})>0$, we have
$\Prob\left( B_i\,\middle|\,\bigcap_{j\in S}\overline{B_j} \right)\le p.$
Then
$\Prob\!\left(\bigcap_{i=1}^m\overline{B_i}\right)>0.$
\end{Lemma}

{\scshape Theorem~\ref{thm:HGPlinear}.}
{\it Let $G$ be a graph of maximum degree $\Delta\ge 2$.  Then}
$$ \HGP(G) < \Delta+ \left\lceil \frac{\Delta^\Delta}{(\Delta-1)^{\Delta-1}}\right\rceil  <(e+1)\Delta.$$

\begin{proof}
Fix the integer $q = \Delta+ \left\lceil \frac{\Delta^\Delta}{(\Delta-1)^{\Delta-1}}\right\rceil$ 
and fix an arbitrary guessing strategy
$f_v:[q]^{N_G(v)}\longrightarrow [q]$ for the vertices $v\in V(G)$.
We shall prove that this strategy is not winning.
Since $q>\Delta$ there exists at least one  proper $q$-coloring, so we may choose a coloring $\Gamma$ uniformly at random from the set of proper $q$-colorings of $G$.
For each vertex $v$, let
$$B_v:= \left\{\Gamma(v)=f_v\bigl(\Gamma|_{N_G(v)}\bigr)\right\}$$
be the bad event that $v$ guesses correctly. Note that the strategy $(f_v)_{v\in V(G)}$ is not winning if we can show that there exists a coloring that avoids all bad events. To that end, we verify the hypothesis of Lemma~\ref{lem:LLLL} with the graph $G$
on the bad events.  Fix $v\in V(G)$ and a set $S\subseteq V(G)\setminus N_G[v]$.
Define the conditioning event
$$\mathcal E_S:=\bigcap_{u\in S}\overline{B_u}$$
that all vertices in $S$ guess wrong, and suppose that $\Prob(\mathcal E_S)>0$.  Let

$$C:=\Gamma|_{V(G)\setminus\{v\}}$$
be the coloring outside $v$.  If $u\in S$, then $u\ne v$ and $uv\notin E(G)$.
Consequently, the event $B_u$ depends only on $\Gamma(u)$ and on the colors in
$N_G(u)$, none of which is $\Gamma(v)$.  Thus every $B_u$ with $u\in S$, and hence
also $\mathcal E_S$, is completely determined by $C$.

Now fix an outside coloring $c$ with $\Prob(C=c)>0$.  Write
$s(c):=\bigl|c(N_G(v))\bigr|$
for the number of distinct colors appearing on the neighbors of $v$.
Conditional on $C=c$, the color $\Gamma(v)$ is uniform on
$
  L_c:=[q]\setminus c(N_G(v)),
$
which has size $q-s(c)$.  Indeed, the proper colorings extending $c$ are
exactly the colorings obtained by assigning to $v$ one color from $L_c$;
each such choice gives exactly one extension, and all proper $q$-colorings
have the same probability.  Since the guess $g_c:=f_v\bigl(c|_{N_G(v)}\bigr)$
is fixed once $c$ is fixed, it follows that
$\Prob(B_v\mid C=c)$ equals $\dfrac{1}{q-s(c)}$ if $g_c\in L_c$, and equals $0$ otherwise.
Therefore
\begin{equation*}
\label{eq:BadeventGivenFixedOutsideColoring}
  \Prob(B_v\mid C=c) \le \frac{1}{q-s(c)} \le \frac{1}{q-d_G(v)}.
\end{equation*}

Averaging this inequality over all outside colorings $c$ yields $\Prob(B_v\mid \mathcal E_S) \le\frac{1}{q-d_G(v)}$, as we now make precise: Let
  $\mathcal G_S :=
  \{c:\Prob(C=c,\mathcal E_S)>0\}$.
Because $\mathcal E_S$ is determined by $C$, every
$c\in\mathcal G_S$ satisfies $\Prob(\mathcal E_S\mid C=c)=1$.
Thus, once $C=c$ is known, conditioning additionally on $\mathcal E_S$
adds no further information.  More explicitly,
$
  \Prob(B_v\mid C=c,\mathcal E_S)
  =
  \frac{\Prob(B_v\cap\mathcal E_S\mid C=c)}
       {\Prob(\mathcal E_S\mid C=c)}
  =\Prob(B_v\mid C=c),
$
because under the condition $C=c$ the event $\mathcal E_S$ occurs with
probability one.  The law of total probability now gives
\begin{align*}
  \Prob(B_v\mid \mathcal E_S)
  &= \sum_{c\in\mathcal G_S}
     \Prob(B_v\mid C=c,\mathcal E_S)\,
     \Prob(C=c\mid \mathcal E_S)\\
  &=\sum_{c\in\mathcal G_S}
     \Prob(B_v\mid C=c)\,
     \Prob(C=c\mid \mathcal E_S)\\
  &\le\sum_{c\in\mathcal G_S}
     \frac{1}{q-d_G(v)} \,
     \Prob(C=c\mid \mathcal E_S)\\
  &=\frac{1}{q-d_G(v)}.
\end{align*}

In particular, $\Prob(B_v\mid \mathcal E_S)
  \le \frac{1}{q-d_G(v)}
  \le \frac{1}{q-\Delta}$.
This holds for every $v$ and every
$S\subseteq V(G)\setminus N_G[v]$ for which the conditioning event has
positive probability.  Hence the bad events satisfy the hypothesis of Lemma~\ref{lem:LLLL}, with the
graph $G$.  By our choice of $q$ we have $\frac{1}{q-\Delta}
  \le
  \frac{(\Delta-1)^{\Delta-1}}{\Delta^\Delta}$, so 
Lemma~\ref{lem:LLLL}  yields
$$\Prob \left(\bigcap_{v\in V(G)}\overline{B_v}\right)>0.$$
Thus against the strategy $(f_v)_{v\in V(G)}$ that was fixed at the beginning, there exists a proper $q$-coloring on which every vertex guesses
incorrectly.  Since the strategy was arbitrary, no winning strategy exists
with $q$ colors, and therefore $\HGP(G)<q$. 
Using that $ \frac{\Delta^\Delta}{(\Delta-1)^{\Delta-1}}
  = \Delta\left(1+\frac{1}{\Delta-1}\right)^{\Delta-1}  <e\Delta$,
it follows that $\HGP(G)<\Delta+e\Delta$.
\end{proof}

\begin{Corollary}\label{cor:cycles}
For every cycle $C$, 
$$4\le \HGP(C) \le 5.$$
\end{Corollary}
\begin{proof}
The lower bound follows because the path on $3$ vertices is a subgraph of $C$, and has proper hat guessing number $4$ due to Theorem~ \ref{Thm:Trees}. The upper bound follows from Theorem~\ref{thm:HGPlinear} with $\Delta=2$.
\end{proof}

\section{Complete graphs}
\label{sec:complete_graphs}

In this section, we determine the proper hat guessing number of the complete graph $K_n$. 
Whilst a linear strategy produces a winning strategy for the classical hat guessing game, 
the proper variant requires a more combinatorial approach. 
Our method is based on the 
construction of regular bipartite graphs and using the existence of a perfect 
matching to generate a guessing strategy. 
We will use the existing literature on perfect matchings between the middle layers of the Boolean lattice to give explicit constructions of winning strategies.
\newline

In the proper variant game on $K_n$, the hat configuration visible from any vertex can be interpreted as a string of length $n-1$, whose entries are all distinct elements of the set $[q]$ of colors.
We design a guessing strategy for $K_n$ with $2n-1$ colors based on a simple observation. Let $[2n-1]$ be the set of colors. Let $\mathcal{S}_{r}$ denote the set of all strings of length $r$ consisting of distinct symbols from $[2n-1]$. The cardinality of both $\mathcal{S}_n$ and $[n] \times \mathcal{S}_{n-1}$ is $\frac{(2n-1)!}{(n-1)!}$.
The following lemma provides a structure for our guessing strategy.

\begin{Lemma}
 \label{Lm:Perfect matching}
For all $n \geq 1$, there exists a bijection $f : \mathcal{S}_n \rightarrow [n] \times \mathcal{S}_{n-1}$ such that for a string $S \in \mathcal{S}_n$, if $f(S) = (r, S^*)$, then $S^*$ is obtained by deleting the $r$\th entry of $S$. 
\end{Lemma}

\begin{proof}

    Let $G = (L \sqcup R, E)$ be a bipartite graph, where $L$ consists of the strings in $\mathcal{S}_n$, $R$ consists of the elements of $[n] \times \mathcal{S}_{n-1}$, and there exists an edge between two vertices $S \in L$ and $(r, S^*) \in R$ if deleting the $r$\th entry $S_r$ of $S$ from the left yields the string $S^*$. We recall that $|L| = |R|$.
\newline

    For every $S \in L$, the degree of $S$ in $G$ is $n$, corresponding to the $n$ entries in $S$ that can be deleted to yield a string in $\mathcal{S}_{n-1}$. Similarly, the degree of each $(r, S^*) \in R$ is $n$, corresponding to the $n$ remaining entries in $[2n-1]$ that could have been deleted as the $r$\th entry of an element in $\mathcal{S}_n$ to form $S^*$. Thus, $G$ is an $n$-regular bipartite graph. By Hall's Matching Theorem, there exists a perfect matching $\mathcal{M} \subseteq E$.  
\newline

    To each $S \in \mathcal{S}_n$, letting $f(S)$ be the unique vertex in $R$ adjacent to $S$ in $\mathcal{M}$ yields a function that satisfies the desired properties. 
\end{proof}

{\scshape Theorem \ref{Thm:CompleteGraphs}.}
 {\it For every integer $n \geq 2$,
 \[\HGP(K_n) = 2n - 1.\]}

\begin{proof}
First note that $\HGP(K_n) \leq n + \chi(K_n) - 1 = 2n-1$ by Lemma~\ref{Lem:ChromaticNumber}.
We now prove that $\HGP(K_n) \geq 2n-1$ by showing that a winning strategy with $2n - 1$ colors exists. 

Let the vertices of $K_n$ be numbered from $1$ to $n$.
We identify a coloring $C$ of $K_n$ with the $n$-tuple $(c_1,\dots,c_n)$ where $c_i$ is the color of vertex $i$.
Let $f$ be the function defined in Lemma~\ref{Lm:Perfect matching}.
We use the following guessing strategy for each vertex $i$: if $i$ sees colors $C^*_i = (c_1,\dots,c_{i-1},c_{i+1},\dots,c_n)$ on its neighbors, $i$ guesses its hat color to be the $i$\textsuperscript{th} entry of $f^{-1}((i,C^*_i))$.

Now suppose that $K_n$ is colored using $C$, and consider $f(C) = (j, C^*_j)$.
Then $j$ guesses the $j$\textsuperscript{th} entry of $f^{-1}(j,C^*_j) = f^{-1}(f(C)) = C$, which is $c_j$, hence $j$ guesses correctly.
Therefore, our strategy is winning.
\end{proof}

The proof of Theorem \ref{Thm:CompleteGraphs} is non-constructive, since we use Hall's Matching Theorem to ensure that the graph $G$ from the proof of Lemma \ref{Lm:Perfect matching} has a perfect matching.
We can however give explicit constructions of perfect matchings for $G$.
Let $\binom{[n]}k$ denote the set of $k$-element subsets of $[n] = \{1,\dots,n\}$.
Consider the bipartite graph $H$ whose vertex classes are $\binom{[2n-1]}{n-1}$ and $\binom{[2n-1]}n$, where two sets $A$ and $B$ are adjacent if and only if $A \subset B$ or $B \subset A$.
Consider the maps
\begin{align*}
 f_L:& \mathcal S_n \to \binom{[2n-1]}{n}: (c_1,\dots,c_n) \mapsto \{c_1,\dots,c_n\}, \\
 f_R:& [n] \times \mathcal S_{n-1} \to \binom{[2n-1]}{n-1}: (i,(c_1,\dots,c_{n-1})) \mapsto \{c_1,\dots,c_{n-1}\}.
\end{align*}

This defines a map $f$ from the vertices of $G$ to the vertices of $H$ with $f(v) = f_L(v)$ if $v \in \mathcal S_n$ and $f(v) = f_R(v)$ if $v \in [n] \times \mathcal S_{n-1}$.

\begin{Lemma}
 The pre-image of a perfect matching in $H$ is a perfect matching in $G$.
\end{Lemma}

\begin{proof}
 An edge of $G$ is of the form
 \[
  \{(c_1,\dots,c_n), (i,(c_1,\dots,c_{i-1},c_{i+1},\dots,c_n))\}
 \]
 and $f$ maps it to the edge
 \[
  \{\{c_1,\dots,c_n\}, \{c_1,\dots,c_n\} \setminus \{c_i\}\}
 \]
 of $H$.
 This proves that $f$ is a graph homomorphism from $G$ to $H$.

 Now suppose that $\mathcal M$ is a perfect matching in $H$.
 Take a vertex $v$ in $G$.
 Then there is a unique edge $e \in \mathcal M$ that covers $f(v)$, and $e = \{A,B\}$ for some $A = \{c_1,\dots,c_n\}$ and $B = A \setminus \{c_n\}$.
 Thus,
 \[
  f_L^{-1}(A) = \{ (c_{\sigma(1)}, \dots, c_{\sigma(n)}) \: : \: \sigma \in \Sym(n) \},
 \]
 where $\Sym(n)$ denotes the symmetric group on $[n]$.
 Any element $(c_{\sigma(1)}, \dots, c_{\sigma(n)}) \in f_L^{-1}(A)$ has a unique neighbor in $f_R^{-1}(B)$, namely
 \[
  (i,(c_{\sigma(1)}, \dots, c_{\sigma(i-1)}, c_{\sigma(i+1)}, \dots, c_{\sigma(n)} )),
 \]
 where $i = \sigma^{-1}(n)$.
 Similarly, every element of $f_R^{-1}(B)$ has a unique neighbor in $f_L^{-1}(A)$.
 Hence, there is a unique edge of $G$ in $f^{-1}(e)$ that covers $v$.
 This proves that $f^{-1}(\mathcal M)$ is a perfect matching of $G$.
\end{proof}

Thus, any explicit perfect matching of $H$ gives rise to an explicit perfect matching of $G$, and hence an explicit winning strategy for the proper hat guessing game on $K_n$.
There are several constructions of perfect matchings of $H$, see e.g.\ \cite{Kierstead1988} for a direct construction, and \cite{Greene1976} for a construction induced by symmetric chain decompositions of the Boolean lattice.
We will describe one construction from Greene and Kleitman \cite{Greene1976}.

Let $q = 2n-1$ and let $[q] = \{1, 2, \dots, q\}$ be the set of possible colors on the graph $K_n$. 
For a given vertex $v_i$, assign a right parenthesis to each of the colors in $[q]$ which is the coloring of a neighbor of $v_i$.
Assign a left parenthesis to each of the remaining colors, and arrange them in order.
The strategy of $v_i$ is to guess the color corresponding to this left-most unmatched left parenthesis.
For example, if $n=4$, $q=7$, and the colors of the neighbors of $v_i$ are 2, 5, and 7, then we have:
\begin{align*}
 &\texttt{1 2 3 4 5 6 7}\\
 &\texttt{( ) ( ( ) ( )}
\end{align*}
Thus, in this case, $v_i$ will guess color $3$.

\section{Trees}
\label{sec:trees}

In this section, we will determine the proper hat guessing number for any arbitrary tree. 
In \cite{Alon2020} it was shown that the ordinary hat guessing number of a tree on at least $2$ vertices is $2$, which by Proposition~\ref{Lem:BoundFromOrdinaryHG} implies that the proper hat guessing number is at most $5$. 
We will prove that for all trees on at least $3$ vertices, the proper hat guessing number is in fact equal to $4$.

A tree on $2$ vertices is just an edge, and hence by Theorem~\ref{Thm:CompleteGraphs} it has proper hat guessing number equal to $3$.
We will restrict our consideration to trees containing at least $3$ vertices.
Note that up to isomorphism, there is a unique tree on exactly $3$ vertices.
We first determine its proper hat guessing number. 

\begin{Lemma}\label{Lem:P3}
 For the path $P_3$ on $2$ edges and $3$ vertices, $\HGP(P_3) = 4$.
\end{Lemma}

\begin{proof}
 Let $P_3 = (\{v_1, v_2, v_3\},\{\{v_1,v_2\},\{v_2,v_3\}\})$.
 Using Lemma \ref{Lem:ChromaticNumber}, we see that $\HGP(P_3) \le 3 + \chi(P_3) - 1 = 4$.
 Now, we present a winning $4$-color guessing strategy for $P_3$.
 Here, we will let the 4 colors be denoted as the elements in $(\mathbb{Z} / 2\mathbb{Z}) \times (\mathbb{Z} / 2 \mathbb{Z})$.
 Define
 \begin{align*}
  f_{v_1}((x_2, y_2)) = (x_2 + 1, y_2), &&
  f_{v_2}((x_1, y_1), (x_3, y_3)) = (x_3 + 1, y_1 + 1), &&
  f_{v_3}((x_2, y_2)) = (x_2, y_2 + 1),
 \end{align*}
 where $(x_i, y_i)$ is the color assigned to vertex $v_i$ for each $i \in  \{1, 2, 3\}$.
 This strategy is easily checked to be winning.
 Suppose that $v_1$ and $v_3$ both guess incorrectly.
 Then $(x_1,y_1)$ must be distinct from $(x_2,y_2)$ since the coloring is proper, and distinct from $(x_2+1,y_2)$, since $v_1$ guesses incorrectly.
 Hence, $y_2 = y_1 + 1$.
 Similarly, $x_2 = x_3 + 1$.
 We see that $v_2$ guesses correctly.
\end{proof}

In \cite[Theorem 1.8]{Alon2020}, the authors proved that if $G$ is a graph with $\HG(G) \geq 3$, then we can remove vertices of degree $1$ without altering the hat guessing number.
We will prove a similar result for the proper hat guessing number.

\begin{Lemma}\label{Lem:LeafRemoval}
 If $G$ is a graph with $\HGP(G) \geq 5$, and $v$ is a vertex of $G$ of degree 1, then $\HGP(G-v) = \HGP(G)$.
\end{Lemma}

\begin{proof}
Suppose that $G$ is a graph with vertices $\{v_1,\dots,v_n\}$, where $v_1$ has $v_2$ as unique neighbor, and suppose that $q \coloneq \HGP(G) \geq 5$.
Let $\{ f_{v_i} \: : \: i = 1, \dots, n \}$ denote a winning strategy for the proper hat guessing game on $G$ with $q$ colors.
For any proper coloring $c$ of $N(v_2) \setminus \{v_1\}$, i.e.\ the set of neighbors of $v_2$ except for $v_1$, let $g_{v_2}(c)$ be the most frequent value of $f_{v_2}(b)$ where $b$ varies through the colorings of $N(v_2)$ that extend $c$.
If there is no unique most frequent value, then pick an arbitrary one.
We will prove that $\{g_{v_2}\} \cup \{f_{v_i} \: : \: i = 3, \dots, n\}$ is a winning strategy for the proper hat guessing game on $G-v_1$ with $q$ colors.

Let $c$ be a proper coloring of $G - v_1$ for which $v_3, \dots, v_n$ guess incorrectly.
Let $\Ex(c)$ denote the set of proper colorings of $G$ that extend $c$.
Note that we can color $v_1$ with any color in $[q] \setminus \{c(v_2)\}$, so $|\Ex(c)| = q-1$.
For each $b \in \Ex(c)$, the vertices $v_3, \dots, v_n$ guess incorrectly, since the color of $v_1$ does not influence their guess.
Thus, for each $b \in \Ex(c)$, according to strategies $f_{v_1}$ and $f_{v_2}$, either $v_1$ or $v_2$ must guess correctly.
Since the guess of $v_1$ is determined by $b(v_2) = c(v_2)$, there is only one coloring $b \in \Ex(c)$ for which $v_1$ guesses correctly.
Then $v_2$ guesses correctly for at least $q-2$ colorings $b \in \Ex(c)$.
This means that if we fix $c|_{N(v_2) \setminus \{v_1\}}$ and vary the color of $v_1$, there are at least $q-2$ colorings of $v_1$ for which $v_2$ guesses $c(v_2)$ according to strategy $f_{v_2}$.
Note that $q-2 > q/2$, since $q \geq 5$.
Then $g_{v_2}(c|_{N(v_2) \setminus \{v_1\}}) = c(v_2)$ by the definition of $g_{v_2}$.
Hence, $v_2$ guesses correctly for the coloring $c$ of $G - v_1$.
We conclude that $\{g_{v_2}\} \cup \{f_{v_i} \: : \: i = 3, \dots, n\}$ is a winning strategy on $G - v_1$.

Thus, $q \leq \HGP(G-v_1) \leq \HGP(G) = q$, which implies that $\HGP(G - v_1) = \HGP(G)$.
\end{proof}

\begin{Corollary}\label{Cor:LeafRemoval}
 For any graph $G$ containing some vertex $v$ of degree $1$, if $G-v$ contains $P_3$ as a subgraph, then $\HGP(G) = \HGP(G - v)$.
\end{Corollary}

\begin{proof}
 Note that $G-v$, and thus also $G$, contains $P_3$ as a subgraph.
 Therefore, the proper hat guessing numbers of $G$ and $G-v$ are at least $4$.
 If $\HGP(G) \geq 5$, then $\HGP(G - v) = \HGP(G)$ by Lemma \ref{Lem:LeafRemoval}.
 Otherwise, we have $4 \leq \HGP(G-v) \leq \HGP(G) \leq 4$, hence $\HGP(G-v) = \HGP(G) = 4$.
\end{proof}

Now Theorem \ref{Thm:Trees} follows easily.
\newline

{\scshape Theorem~\ref{Thm:Trees}.}
 {\it For every tree $T$ on at least $3$ vertices, $\HGP(T) = 4$.}
\begin{proof}
 We prove this by induction on the number $n$ of vertices of the tree.
 The base case $n=3$ was proven in Lemma~\ref{Lem:P3}.
 The induction step follows from Corollary~\ref{Cor:LeafRemoval}, since every tree has vertices of degree 1.
\end{proof}

\setcounter{section}{4}

\section{Book graphs}
\label{sec:book_graphs}

Let $c_q(G)$ denote the number of proper colorings of a graph $G$ using $q$ colors. For example, the number of proper colorings of the complete graph $K_n$ using $q$ colors is the falling factorial $c_q(K_n) = (q)_n = q(q-1)\dots(q-n+1)$.

Let $k,n$ be positive integers. The book graph $B_{k,n}$ is the complete connection between a clique of size $k$ and an independent set of size $n$. The clique is referred to as the spine of the book graph and the vertices in the independent set are referred to as the pages of the book graph. Given any of the $(q)_n$ proper colorings of the spine using $q$ colors, we may color the pages using the remaining $q-k$ colors in any way to obtain a proper coloring of the book graph. Therefore, $c_q(B_{k,n}) = (q)_k(q-k)^n$. 
We can use this to give a refinement of Lemma~\ref{Lem:ChromaticNumber} for book graphs.

\begin{Lemma}\label{Lem:BookGraphCondition}
    Let $k,n,q$ positive integers such that $q \geq k+1$. We have $\HGP(B_{k,n}) < q$ whenever
    \begin{equation*}
        (q-k-1)^n - k(q-k+1)^{n-1} > 0.
    \end{equation*}
\end{Lemma}
\begin{proof}
    Fix some strategy for the proper hat guessing game on $B_{k,n}$ with $q$ colors. Note that the total number of proper colorings of $B_{k,n}$ using $q$ colors is $c_q(B_{k,n}) = (q)_k(q-k)^n$.

    Consider some spine vertex $v$. 
    Its guess gets fixed by a coloring of the other $k-1+n$ vertices, which induce a subgraph isomorphic to $B_{k-1,n}$. 
    Such a coloring has at most one extension in which $v$ guesses correctly, namely the extension assigning the guessed color to $v$. 
    Therefore there are at most $c_q(B_{k-1,n}) = (q)_{k-1}(q-(k-1))^n$ proper colorings such that the selected spine vertex $v$ guesses correctly.

    Now consider the pages. All of their guesses get fixed by a proper coloring of the spine. For any such coloring there are $q-k-1$ ways to color a page vertices to preserve the proper coloring and make sure the page vertex does not guess correctly. So there are $(q)_k(q-k-1)^n$ proper colorings such that the pages all guess incorrectly. Complementary, there are $c_q(B_{k,n}) - (q)_k(q-k-1)^n$ proper colorings such that some page vertex guesses correctly.

    By the union bound, there are at most 
    \begin{align*}
        k(q)_{k-1}(q-(k-1))^n + c_q(B_{k,n}) - (q)_k(q-k-1)^n \\
        = c_q(B_{k,n}) + (q)_k\big(k(q-(k-1))^{n-1} - (q-k-1)^n\big)
    \end{align*}
    proper colorings such that some vertex guesses correctly. When the condition in the statement is satisfied, then $c_q(B_{k,n})$ is more than the number of proper colorings such that some vertex guesses correctly. So there exists some coloring such that no vertex guesses correctly and thus $\HGP(B_{k,n}) < q$.
\end{proof}

The condition in the above Lemma \ref{Lem:BookGraphCondition} may be satisfied by a sufficiently small value of $q$, when the number of pages $n$ is even. However, we require $q \geq k+1$ to guarantee the existence of a proper coloring. Consequently, $q$ cannot be small enough for the condition to be satisfied in this manner.

Note that for book graphs we have $\HGP(B_{k,n}) < n + 2k + 1$ by Lemma \ref{Lem:ChromaticNumber}. Using the above result, we can show a stronger upper bound holds when the number of pages $n$ is sufficiently larger than the size of the spine $k$.
\newline

{\scshape Theorem \ref{Thm:BookGraphUpperBound}.}
 {\it If $n > e^2 k$, then
 \[
  \HGP(B_{k,n}) < \left( \frac{e^2 k}{n} \right)^{1/n}n + k + 1.
 \]
}
\begin{proof}
 Using Lemma~\ref{Lem:BookGraphCondition}, it suffices to check that
 \[
  \left( \left( \frac{e^2 k}n \right)^{1/n} n + k + 1 - k -1 \right)^n - k \left( \left( \frac{e^2 k}n \right)^{1/n} n + k + 1 - k +1 \right)^{n-1} > 0.
 \]
 This simplifies to checking that
 \[
  e^2 n^{n-1} > \left( \left( \frac{ e^2 k}n \right)^{1/n} n + 2 \right)^{n-1}.
 \]
 Since $\frac{e^2 k}n < 1$, it suffices that $e^2 n^{n-1} > (n+2)^{n-1}$, or equivalently $e^2 > \left( 1 + \frac 2n \right)^{n-1}$.
 This is easily checked to hold since $e^2 > \left( 1 + \frac 2n \right)^n > \left( 1 + \frac 2n \right)^{n-1}$ for every positive number $n$.
\end{proof}

\begin{Theorem}
 \label{Thm:Book:2nd upper bound}
 If $k = \mathcal O(n^{1-\varepsilon})$ for some $\varepsilon > 0$, then $\HGP(B_{k,n}) = \mathcal O(n/\ln n)$.
\end{Theorem}

\begin{proof}
 Write $x = q - k -1$.
 By Lemma~\ref{Lem:BookGraphCondition}, if $x^n > k (x + 2)^{n-1}$, then $\HGP(B_{k,n}) < q$.
 We can rewrite this as $x > k \left( 1 + \frac 2x \right)^{n-1}$.
 Note that
 \[
  \left( 1 + \frac 2x \right)^{n-1}
  = \left( \left( 1 + \frac 2x \right)^x \right)^\frac{n-1}x
  < e^{2\frac{n-1}x} < e^{2\frac nx}.
 \]
 Thus, it suffices that $x \geq k e^{2n/x}$.
 Therefore, we want to show that there exists a constant $c$ such that if we write $x = c n / \ln n$, then
 \[
  c \frac n {\ln n} \geq k e^\frac{2n}{c n / \ln n} = k n^{2/c}
 \]
 holds for $k$ large enough.
 In other words, $c n^{1 - 2/c} \geq k \ln n$.
 If we pick $c > 2 / \varepsilon$, then this holds for $k$ large enough, and we have for large $k$ that $\HGP(B_{k,n}) < c \frac n {\ln n} + k + 1 = \mathcal O(n / \ln n)$.
\end{proof}

We know from \cite[Theorem 1.2]{He2022} that if $k$ is fixed, and $n$ is sufficiently large, then $\HG(B_{k,n}) = 1 + \sum_{m=1}^k m^m$, which is independent of $n$.
Note that $\chi(B_{k,n}) = k+1$.
Hence, by Proposition~\ref{Lem:BoundFromOrdinaryHG}
\[
 \HGP(B_{k,n}) < (k+1)\left( 2 + \sum_{m=1}^k m ^m \right).
\]
In particular, if we fix $k$ and let $n \to \infty$, then $\HGP(B_{k,n})$ is also bounded by a constant.

\section{Small Graphs}
\label{sec:small_graphs}

The regular hat guessing number is known for all graphs on 5 or fewer vertices. In this section we give an overview of the proper hat guessing number for a number of small graphs. We will only consider connected graphs, as the proper hat guessing number of disconnected graphs is the maximum of the proper hat guessing numbers of its connected parts. Furthermore, by Corollary \ref{Cor:LeafRemoval}, we may remove vertices of degree 1, so we will not consider graphs whose hat guessing number can be derived from smaller graphs in this manner. At the end of this section we give an overview of what can be derived about the proper hat guessing numbers of such graphs with $5$ or fewer vertices. First, we present an ILP formulation of the proper hat guessing number in the form of a feasibility problem. 
We have used this ILP to solve the cases of $C_4$ and $K_4 - e$. 

Let $G=(V,E)$ be a graph with $V=\{v_1,\dots,v_n\}$ and fix a positive integer $q$.
Fix, for each $i\in[n]$, an order on the neighborhood $N_G(v_i)$ so that for any
coloring $c=(c_1,\dots,c_n)\in [q]^n$ the neighborhood pattern
\[
P_i(c)\;=\;\bigl(c(x):x\in N_G(v_i)\bigr)
\]
is an \emph{ordered} tuple. Let $\mathcal{C}_q\subseteq [q]^n$ be the set of all proper
colorings of $G$. For each $i\in[n]$ define the set of attainable patterns
\[
\Pi_i \;:=\; \{\,P_i(c): c\in \mathcal{C}_q\,\}.
\]

\paragraph{Variables.}
For each $i\in[n]$, $P\in\Pi_i$, and $g\in [q]\setminus P$,
let $f_i(P,g)\in\{0,1\}$ indicate whether vertex $v_i$ guesses color $g$ upon seeing
pattern $P$.

\paragraph{ILP (feasible iff $\HGP(G)\ge q$).}
\begin{align*}
\textbf{Minimize}\quad & 0 \\[-2pt]
\textbf{Subject to}\quad
& \sum_{g\in [q]\setminus P} f_i(P,g) \;=\; 1
&& \forall\, i\in[n],\ \forall\, P\in \Pi_i \\[4pt]
& \sum_{i=1}^{n} f_i\!\bigl(P_i(c),\, c_i\bigr) \;\ge\; 1
&& \forall\, c=(c_1,\dots,c_n)\in \mathcal{C}_q \\[4pt]
& f_i(P,g)\in\{0,1\}
&& \forall\, i\in[n],\ \forall\, P\in \Pi_i,\ \forall\, g\in [q]\setminus P\,.
\end{align*}

Table \ref{Table:small} contains all connected graphs on $5$ or fewer vertices with minimum degree 2 when applicable and what is known about their proper hat guessing number. Note that for the proper hat guessing number of a subgraph $H \subseteq G$ we have $\HGP(H) \leq \HGP(G)$ by Lemma \ref{Lem:Subgraphs}.

\begin{table} 
 \centering
\begin{tabular}{c c c c}
\hline
Name & $G$ & $\HGP(G)$ & Proof \\
\hline
\multicolumn{3}{c}{\textbf{1 vertex}}\\
$K_1$ &
\begin{tikzpicture}[baseline=-0.5ex, scale=0.9]
    \node[circle,fill,inner sep=1.4pt] at (0,0) {};
\end{tikzpicture}
& $1$ & Theorem \ref{Thm:CompleteGraphs} \\[6pt]
\hline
\multicolumn{3}{c}{\textbf{2 vertices}}\\
$K_2$ &
\begin{tikzpicture}[baseline=-0.5ex, scale=0.9]
    \node[circle,fill,inner sep=1.4pt] (a) at (-0.5,0) {};
    \node[circle,fill,inner sep=1.4pt] (b) at (0.5,0) {};
    \draw (a)--(b);
\end{tikzpicture}
& $3$ & Theorem \ref{Thm:CompleteGraphs} \\[6pt]
\hline
\multicolumn{3}{c}{\textbf{3 vertices}}\\
$P_3$ &
\begin{tikzpicture}[baseline=-0.5ex, scale=0.9]
  \node[circle,fill,inner sep=1.4pt] (a) at (0,0) {};
  \node[circle,fill,inner sep=1.4pt] (b) at (1,0) {};
  \node[circle,fill,inner sep=1.4pt] (c) at (2,0) {};
  \draw (a)--(b)--(c);
\end{tikzpicture}
& $4$ & Lemma \ref{Lem:P3} \\[6pt]
$K_3$ &
\begin{tikzpicture}[baseline=-0.5ex, scale=0.9]
  \node[circle,fill,inner sep=1.4pt] (a) at (0,0) {};
  \node[circle,fill,inner sep=1.4pt] (b) at (1,0) {};
  \node[circle,fill,inner sep=1.4pt] (c) at (0.5,0.866) {};
  \draw (a)--(b)--(c)--(a);
\end{tikzpicture}
& $5$ & Theorem \ref{Thm:CompleteGraphs} \\[6pt]
\hline
\multicolumn{3}{c}{\textbf{4 vertices}}\\
$C_4$ &
\begin{tikzpicture}[baseline=-0.5ex, scale=0.9]
  \node[circle,fill,inner sep=1.4pt] (a) at (0,0) {};
  \node[circle,fill,inner sep=1.4pt] (b) at (1,0) {};
  \node[circle,fill,inner sep=1.4pt] (c) at (1,1) {};
  \node[circle,fill,inner sep=1.4pt] (d) at (0,1) {};
  \draw (a)--(b)--(c)--(d)--(a);
\end{tikzpicture}
& $4$ & ILP \\[6pt]
$K_4 - e$ &
\begin{tikzpicture}[baseline=-0.5ex, scale=0.9]
  \node[circle,fill,inner sep=1.4pt] (a) at (0,0) {};
  \node[circle,fill,inner sep=1.4pt] (b) at (1,0) {};
  \node[circle,fill,inner sep=1.4pt] (c) at (1,1) {};
  \node[circle,fill,inner sep=1.4pt] (d) at (0,1) {};
  \draw (a)--(b)--(c)--(d)--(a) (b)--(d);
\end{tikzpicture}
& $6$ & ILP \\[6pt]
$K_4$ &
\begin{tikzpicture}[baseline=-0.5ex, scale=0.9]
  \node[circle,fill,inner sep=1.4pt] (a) at (0,0) {};
  \node[circle,fill,inner sep=1.4pt] (b) at (1,0) {};
  \node[circle,fill,inner sep=1.4pt] (c) at (1,1) {};
  \node[circle,fill,inner sep=1.4pt] (d) at (0,1) {};
  \draw (a)--(b)--(c)--(d)--(a) (a)--(c) (b)--(d);
\end{tikzpicture}
& $7$ & Theorem \ref{Thm:CompleteGraphs} \\[6pt]
\hline
\multicolumn{3}{c}{\textbf{5 vertices}}\\
$C_5$ &
\begin{tikzpicture}[baseline=-0.5ex, scale=0.6]
  \node[circle,fill,inner sep=1.4pt] (a) at ({sin(0)},{cos(0)}) {};
  \node[circle,fill,inner sep=1.4pt] (b) at ({sin(72)},{cos(72)}) {};
  \node[circle,fill,inner sep=1.4pt] (c) at ({sin(144)},{cos(144)}) {};
  \node[circle,fill,inner sep=1.4pt] (d) at ({sin(216)},{cos(216)}) {};
  \node[circle,fill,inner sep=1.4pt] (e) at ({sin(288)},{cos(288)}) {};
  \draw (a)--(b)--(c)--(d)--(e)--(a);
\end{tikzpicture}
& $4 \leq \ldots \leq 5$ & Lemma \ref{Lem:P3}, Corollary~\ref{cor:cycles}\\[6pt]
$K_{2,3}$ &
\begin{tikzpicture}[baseline=-0.5ex, scale=0.6]
  \node[circle,fill,inner sep=1.4pt] (a) at ({sin(0)},{cos(0)-0.5}) {};
  \node[circle,fill,inner sep=1.4pt] (b) at ({sin(72)},{cos(72)}) {};
  \node[circle,fill,inner sep=1.4pt] (c) at ({sin(144)},{cos(144)}) {};
  \node[circle,fill,inner sep=1.4pt] (d) at ({sin(216)},{cos(216)}) {};
  \node[circle,fill,inner sep=1.4pt] (e) at ({sin(288)},{cos(288)}) {};
  \draw (c)--(a) (c)--(b) (c)--(e) (d)--(a) (d)--(b) (d)--(e);
\end{tikzpicture}
& $5$ & ILP \\[6pt]
$\text{Hourglass}$ &
\begin{tikzpicture}[baseline=-0.5ex, scale=0.9]
  \node[circle,fill,inner sep=1.4pt] (a) at (0,0) {};
  \node[circle,fill,inner sep=1.4pt] (b) at (1,0) {};
  \node[circle,fill,inner sep=1.4pt] (c) at (1,1) {};
  \node[circle,fill,inner sep=1.4pt] (d) at (0,1) {};
  \node[circle,fill,inner sep=1.4pt] (e) at (0.5,0.5) {};
  \draw (a)--(c)--(d)--(b)--(a);
\end{tikzpicture}
& $5 \leq \ldots \leq 7$ & Theorem \ref{Thm:CompleteGraphs}, Lemma \ref{Lem:ChromaticNumber} \\[6pt]
$\text{House}$ &
\begin{tikzpicture}[baseline=-0.5ex, scale=0.9]
  \node[circle,fill,inner sep=1.4pt] (a) at (0,0) {};
  \node[circle,fill,inner sep=1.4pt] (b) at (1,0) {};
  \node[circle,fill,inner sep=1.4pt] (c) at (1,1) {};
  \node[circle,fill,inner sep=1.4pt] (d) at (0,1) {};
  \node[circle,fill,inner sep=1.4pt] (e) at (0.5,1.5) {};
  \draw (a)--(b)--(c)--(d)--(a) (c)--(e)--(d);
\end{tikzpicture}
& $5 \leq \ldots \leq 6$ & Theorem \ref{Thm:CompleteGraphs}, Remark \ref{Rmk:ChromaticNumber} \\[6pt]
$B_{2,3}$ &
\begin{tikzpicture}[baseline=-0.5ex, scale=0.6]
  \node[circle,fill,inner sep=1.4pt] (a) at ({sin(0)},{cos(0)-0.5}) {};
  \node[circle,fill,inner sep=1.4pt] (b) at ({sin(72)},{cos(72)}) {};
  \node[circle,fill,inner sep=1.4pt] (c) at ({sin(144)},{cos(144)}) {};
  \node[circle,fill,inner sep=1.4pt] (d) at ({sin(216)},{cos(216)}) {};
  \node[circle,fill,inner sep=1.4pt] (e) at ({sin(288)},{cos(288)}) {};
  \draw (c)--(a) (c)--(b) (c)--(e) (d)--(a) (d)--(b) (d)--(e) (c)--(d);
\end{tikzpicture}
& $6 \leq \ldots \leq 7$ & ILP, Lemma \ref{Lem:BookGraphCondition} \\[6pt]
$5$-Fan &
\begin{tikzpicture}[baseline=-0.5ex, scale=0.6]
  \node[circle,fill,inner sep=1.4pt] (a) at ({sin(0)},{cos(0)}) {};
  \node[circle,fill,inner sep=1.4pt] (b) at ({sin(72)},{cos(72)}) {};
  \node[circle,fill,inner sep=1.4pt] (c) at ({sin(144)},{cos(144)}) {};
  \node[circle,fill,inner sep=1.4pt] (d) at ({sin(216)},{cos(216)}) {};
  \node[circle,fill,inner sep=1.4pt] (e) at ({sin(288)},{cos(288)}) {};
  \draw (a)--(b)--(c)--(d)--(e)--(a) (c)--(a)--(d);
\end{tikzpicture}
& $5 \leq \ldots \leq 7$ & Theorem \ref{Thm:CompleteGraphs}, Lemma \ref{Lem:ChromaticNumber} \\[6pt]
$\text{Broken Wheel}$ &
\begin{tikzpicture}[baseline=-0.5ex, scale=0.9]
  \node[circle,fill,inner sep=1.4pt] (a) at (0,0) {};
  \node[circle,fill,inner sep=1.4pt] (b) at (1,0) {};
  \node[circle,fill,inner sep=1.4pt] (c) at (1,1) {};
  \node[circle,fill,inner sep=1.4pt] (d) at (0,1) {};
  \node[circle,fill,inner sep=1.4pt] (e) at (0.5,0.5) {};
  \draw (a)--(b)--(c)--(d)--(a) (a)--(c) (e)--(d);
\end{tikzpicture}
& $6$ & ILP, Remark \ref{Rmk:ChromaticNumber} \\[6pt]
$\text{$W_4$}$ &
\begin{tikzpicture}[baseline=-0.5ex, scale=0.9]
  \node[circle,fill,inner sep=1.4pt] (a) at (0,0) {};
  \node[circle,fill,inner sep=1.4pt] (b) at (1,0) {};
  \node[circle,fill,inner sep=1.4pt] (c) at (1,1) {};
  \node[circle,fill,inner sep=1.4pt] (d) at (0,1) {};
  \node[circle,fill,inner sep=1.4pt] (e) at (0.5,0.5) {};
  \draw (a)--(b)--(c)--(d)--(a) (a)--(c) (d)--(b);
\end{tikzpicture}
& $6 \leq \ldots \leq 7$ & ILP, Lemma \ref{Lem:ChromaticNumber} \\[6pt]
$\text{Closed House}$ &
\begin{tikzpicture}[baseline=-0.5ex, scale=0.9]
  \node[circle,fill,inner sep=1.4pt] (a) at (0,0) {};
  \node[circle,fill,inner sep=1.4pt] (b) at (1,0) {};
  \node[circle,fill,inner sep=1.4pt] (c) at (1,1) {};
  \node[circle,fill,inner sep=1.4pt] (d) at (0,1) {};
  \node[circle,fill,inner sep=1.4pt] (e) at (0.5,1.5) {};
  \draw (a)--(b)--(c)--(d)--(a) (a)--(c)--(e)--(d)--(b);
\end{tikzpicture}
& $7$ & Theorem \ref{Thm:CompleteGraphs}, Remark \ref{Rmk:ChromaticNumber} \\[6pt]
$K_5-e$ &
\begin{tikzpicture}[baseline=-0.5ex, scale=0.6]
  \node[circle,fill,inner sep=1.4pt] (a) at ({sin(0)},{cos(0)}) {};
  \node[circle,fill,inner sep=1.4pt] (b) at ({sin(72)},{cos(72)}) {};
  \node[circle,fill,inner sep=1.4pt] (c) at ({sin(144)},{cos(144)}) {};
  \node[circle,fill,inner sep=1.4pt] (d) at ({sin(216)},{cos(216)}) {};
  \node[circle,fill,inner sep=1.4pt] (e) at ({sin(288)},{cos(288)}) {};
  \draw (a)--(b)--(c) (d)--(e)--(a) (a)--(c)--(e)--(b)--(d)--(a);
\end{tikzpicture}
& $7 \leq \ldots \leq 8$ & Theorem \ref{Thm:CompleteGraphs}, Lemma \ref{Lem:ChromaticNumber} \\[6pt]
$K_5$ &
\begin{tikzpicture}[baseline=-0.5ex, scale=0.6]
  \node[circle,fill,inner sep=1.4pt] (a) at ({sin(0)},{cos(0)}) {};
  \node[circle,fill,inner sep=1.4pt] (b) at ({sin(72)},{cos(72)}) {};
  \node[circle,fill,inner sep=1.4pt] (c) at ({sin(144)},{cos(144)}) {};
  \node[circle,fill,inner sep=1.4pt] (d) at ({sin(216)},{cos(216)}) {};
  \node[circle,fill,inner sep=1.4pt] (e) at ({sin(288)},{cos(288)}) {};
  \draw (a)--(b)--(c)--(d)--(e)--(a) (a)--(c)--(e)--(b)--(d)--(a);
\end{tikzpicture}
& $9$ & Theorem \ref{Thm:CompleteGraphs} \\[6pt]
\end{tabular}
 \caption{Proper hat guessing numbers on graphs on at most 5 vertices.}
  \label{Table:small}
\end{table}

\section{Conclusion and Future Work}

In this paper, we introduced a new variant of the hat guessing game on a graph.
We developed the basic theory, in particular determining the proper hat guessing number for complete graphs and trees.
In large part, parallels can be drawn to the study of the ordinary hat guessing number, but at some points the theory differs.
We end this paper with a non-exhaustive list of interesting questions that remain.

\begin{Problem}
 There are variations on the ordinary hat guessing game where e.g.\ players can have multiple guesses, or the number of possible hat colors is not the same for all players.
 Investigate these variations of the proper hat guessing game.
\end{Problem}

We remark that if each player is allowed $m$ guesses, then the analog of Lemma~\ref{Lem:ChromaticNumber} tells us that a winning game on $G$ uses at most $mn + \chi(G) - 1$ colors.
For the complete graph, this becomes $(m+1)n - 1$.
A small adaptation of the proof of Theorem~\ref{Thm:CompleteGraphs}, again relying on Hall's Matching Theorem, shows that there exists a winning strategy on $K_n$ with $(m+1)n-1$ colors and $m$ guesses.

\bigskip

We proved that $\HGP(G)$ is upper bounded by a constant times the maximum degree $\Delta(G)$.
A natural question is to determine the optimal constant.

\begin{Problem}
What is the smallest
constant $c>0$ such that $\HGP(G)\leq c \cdot \Delta(G)$ for all graphs $G$ with sufficiently large maximum degree? Theorem~\ref
{thm:HGPlinear} and Theorem~\ref{Thm:CompleteGraphs} imply that $2 \le c \le e+1$.
\end{Problem}

Another problem that remains is determining the proper hat guessing number of basic graph classes, such as cycles, wheel graphs, windmill graphs, and many more.
We touched on the proper hat guessing number for book graphs, but the problem is far from resolved.
We state the most important open question.

\begin{Problem}
 If we fix an integer $k$ and let $n \to \infty$, then we saw that $\HGP(B_{k,n})$ is bounded.
 Determine the maximum value of $\HGP(B_{k,n})$.
 We know that it is somewhere between $1 + \sum_{m=1}^k m^m$ and $(k+1)(2 + \sum_{m=1}^k m^m)$.
\end{Problem}

We can also wonder how the hat guessing number drops if we start with a complete graph and remove edges.

\begin{Problem}
 What is the proper hat guessing number of $K_n - e$, the complete graph minus an edge?
 We know that it is either $2n-2$ or $2n-3$.
\end{Problem}

From Remark \ref{Rmk:ChromaticNumber}, we see that if we fix a vertex $v$ of $K_n$ and remove at least 2 edges incident with $v$ from $K_n$, then the proper hat guessing number of the resulting graph is at most $2n-3$, and since it has $K_{n-1}$ as a subgraph it is actually equal to $2n-3$.

\section*{Acknowledgments}
This project began during the 2025 Polymath Jr.\ program. 
The program is partially supported by the National Science Foundation under award DMS-2218374. 
We thank the organizers and participants of this program.
Sam Adriaensen is supported by grant number 12A3Y25N of the Research Foundation Flanders (FWO).
This research was carried out as part of the Scientific Research Network W003324N
of Research Foundation Flanders (FWO).
We thank Raphael Mario Steiner for valuable insights and discussion.


\begin{thebibliography}{10}

\bibitem{Alon2020}
N.~Alon, O.~Ben-Eliezer, C.~Shangguan, and I.~Tamo.
\newblock The hat guessing number of graphs.
\newblock {\em J. Comb. Theory Ser. B}, 144:119--149, 2020.

\bibitem{Bradshaw2022}
P.~Bradshaw.
\newblock On the hat guessing number of a planar graph class.
\newblock {\em J. Comb. Theory Ser. B}, 156:174--193, 2022.

\bibitem{farnik}
M.~Farnik.
\newblock {\em A hat guessing game}.
\newblock PhD thesis, Jagiellonian University, 2015.



\bibitem{ErdosSpencer}
P.~Erd\H{o}s and J.~Spencer,
\emph{Lopsided Lov\'{a}sz local lemma and Latin transversals},
Discrete Applied Mathematics 30 (1991), 151--154.

\bibitem{ScottSokal}
A.~D. Scott and A.~D. Sokal,
\emph{The repulsive lattice gas, the independent-set polynomial, and the
Lov\'{a}sz local lemma},
Journal of Statistical Physics 118 (2005), 1151--1261,
DOI: 10.1007/s10955-004-2055-4.

\bibitem{KnuthBook}
D.~E. Knuth,
\emph{The Art of Computer Programming, Volume~4, Fascicle~6:
Satisfiability},
Addison--Wesley, Boston, MA, 2015.




\bibitem{Feige2004}
U.~Feige.
\newblock You can leave your hat on (if you guess its color).
\newblock Technical Report MCS04-03, The Weizmann Institute of Science, 2004.

\bibitem{Greene1976}
C.~Greene and D.~J. Kleitman.
\newblock Strong versions of {S}perner's theorem.
\newblock {\em J. Comb. Theory Ser. A}, 20(1):80--88, 1976.

\bibitem{He2022}
X.~He, Y.~Ido, and B.~Przybocki.
\newblock Hat guessing on books and windmills.
\newblock {\em Electron. J. Combin.}, 29(1):Paper No. 1.12, 19, 2022.

\bibitem{Kierstead1988}
H.~Kierstead and W.~Trotter.
\newblock Explicit matchings in the middle levels of the boolean lattice.
\newblock {\em Order}, 5(2):163--171, June 1988.

\bibitem{Knierim2023}
C.~Knierim, A.~Martinsson, and R.~Steiner.
\newblock Hat guessing numbers of strongly degenerate graphs.
\newblock {\em SIAM Journal on Discrete Mathematics}, 37(2):1331--1347, June 2023.

\bibitem{latyshev2023}
A.~Latyshev and K.~Kokhas.
\newblock Hat guessing number of planar graphs is at least 22.
\newblock {\em Discrete Mathematics}, 347:113820, 2024.

\bibitem{Szczechla2017}
W.~Szczechla.
\newblock The three color hat guessing game on cycle graphs.
\newblock {\em Electron. J. Combin.}, 24(1):Paper No. 1.37, 19, 2017.

\bibitem{Winkler2001}
P.~Winkler.
\newblock Games people don't play.
\newblock In D.~Wolfe and T.~Rodgers, editors, {\em Puzzlers' Tribute: A Feast for the Mind}, pages 301--313. A K Peters/CRC Press, Natick, MA, 2001.

\end{thebibliography}
\end{document}